\documentclass[11pt,reqno]{amsart}
\usepackage[T1]{fontenc}
\usepackage{lmodern}
\usepackage{amsmath,amssymb,mathtools}
\usepackage[margin=29mm]{geometry}
\usepackage{microtype}
\usepackage[hidelinks]{hyperref}
\usepackage{bookmark}
\numberwithin{equation}{section}
\theoremstyle{plain}
\newtheorem{theorem}{Theorem}[section]
\newtheorem{proposition}[theorem]{Proposition}
\newtheorem{lemma}[theorem]{Lemma}
\newtheorem{corollary}[theorem]{Corollary}
\theoremstyle{remark}

\newcommand{\N}{\mathbb N}
\newcommand{\Z}{\mathbb Z}
\newcommand{\R}{\mathbb R}
\newcommand{\T}{\mathbb T}
\newcommand{\PP}{\mathbb P}
\newcommand{\EE}{\mathbb E}
\newcommand{\cA}{\mathcal A}
\newcommand{\cF}{\mathcal F}
\newcommand{\cI}{\mathcal I}
\newcommand{\cT}{\mathcal T}

\allowdisplaybreaks[1]

\title[The Erd\H{o}s similarity problem for null sequences]
{The Erd\H{o}s similarity problem for null sequences with adjacent ratios bounded below}
\author{Kan Jiang}
\address{Department of Mathematics, Ningbo University,
Ningbo 315211, P. R. China}
\email{jiangkan@nbu.edu.cn}

\author{Lifeng Xi}
\address{Department of Mathematics, Ningbo University,
Ningbo 315211, P. R. China}
\email{xilifeng@nbu.edu.cn; xilifengningbo@yahoo.com}
\thanks{This work is supported by the National Natural Science Foundation
of China (Nos.~12671112, 12571096, 12171259, 11771226).}

\date{}
\subjclass[2020]{28A80}
\keywords{Erd\H{o}s similarity conjecture, affine copy, null sequence,
Lebesgue measure, probabilistic construction}
\hypersetup{
  pdftitle={The Erdos similarity problem for null sequences with adjacent ratios bounded below},
  pdfauthor={Kan Jiang and Lifeng Xi},
  pdfsubject={2020 Mathematics Subject Classification: 28A80},
  pdfkeywords={Erdos similarity conjecture, affine copy, null sequence, Lebesgue measure, probabilistic construction}
}

\begin{document}

\begin{abstract}
Let $(a_n)$ be a positive null sequence with
$\liminf_{n\to\infty}a_{n+1}/a_n>0$. For every $\eta>0$, we construct
a compact set $E\subset[0,1]$ with Lebesgue measure greater than
$1-\eta$ containing no nontrivial affine copy of $\{a_n:n\ge1\}$.
No monotonicity is assumed. The proof adapts a finite-tree construction
for geometric sequences to grids chosen from a subsequence with
two-sided ratio bounds. We also obtain one such set for any prescribed
countable family of sequences, with infinitely many distinct points
omitted from every affine copy.
\end{abstract}

\maketitle

\section{Introduction}

A set $A\subset\R$ is \emph{measure universal} if every Lebesgue
measurable set of positive measure contains a nontrivial affine copy
\[
 x+sA=\{x+sa:a\in A\},\qquad x\in\R,\quad s\in\R\setminus\{0\}.
\]
The Erd\H{o}s similarity conjecture states that no infinite set is
measure universal; see Erd\H{o}s~\cite{Erdos1974}. In contrast, every
finite set is measure universal. Indeed, if $F$ is finite and $B$ has finite positive
measure, continuity of translations in $L^1$ gives
\[
 m\left(\bigcap_{a\in F}(B-sa)\right)\longrightarrow m(B)>0
 \qquad(s\to0),
\]
where $m$ denotes Lebesgue measure. For accounts of the problem, see
Svetic~\cite{Svetic2001} and Jung, Lai and
Mooroogen~\cite{JungLaiMooroogen2025}.

We consider positive sequences tending to zero whose adjacent ratios
are eventually bounded away from zero. The following is our main
result.

\begin{theorem}\label{thm:main}
Let $a_n>0$, $a_n\to0$, and
\begin{equation}\label{eq:main-hypothesis}
 \liminf_{n\to\infty}\frac{a_{n+1}}{a_n}>0.
\end{equation}
For every $\eta\in(0,1)$ there is a compact set $E\subset[0,1]$ such that
$m(E)>1-\eta$ and
\begin{equation}\label{eq:main-conclusion}
 \forall x\in\R\ \forall s\in\R\setminus\{0\}\ \exists n\ge1:
 \qquad x+sa_n\notin E.
\end{equation}
In particular, $\{a_n:n\ge1\}$ is not measure universal.
\end{theorem}

The sequence need not be decreasing, and its adjacent ratios need not
converge. The set $E$ depends on the given sequence and on $\eta$.

Falconer~\cite{Falconer1984} and Eigen~\cite{Eigen1985} proved the
conjecture for positive decreasing null sequences with
$a_{n+1}/a_n\to1$. Humke and Laczkovich~\cite{HumkeLaczkovich1998}
gave a finite covering characterization and further criteria involving
relative gaps. Kolountzakis~\cite{Kolountzakis1997} used a probabilistic
construction. Chleb\'ik~\cite[Theorem~15]{Chlebik2015} proved the
following translation-invariant criterion: a bounded infinite set is
not measure universal if it contains finite subsets
$x_1>\cdots>x_m$, with $m\to\infty$, such that
\begin{equation}\label{eq:relative-gap}
 -\log\left(\frac{\min_{i<m}(x_i-x_{i+1})}{x_1-x_m}\right)=o(m).
\end{equation}
This criterion does not include geometric sequences. For $m$ distinct
powers of $q\in(0,1)$, arranged in decreasing order, the normalized
minimum gap is at most $q^{m-2}/(1-q)$.

Bourgain~\cite{Bourgain1987} proved that the sum of three infinite
sets is not measure universal. More recently, Mora Cuellar, Iosevich,
Kulkarni, Rojas Aravena and Yavicoli~\cite[Corollary~6(b)]{Mora2026}
considered two-fold sums and differences. Under the lower ratio
condition~\eqref{eq:main-hypothesis}, for a decreasing sequence,
they proved that $\{a_n:n\ge1\}+A$ and $\{a_n:n\ge1\}-A$ are not
measure universal for every infinite set $A$. These results do not
imply nonuniversality of the sequence itself.

The bi-Lipschitz version has a different answer. Feng, Lai and
Xiong~\cite[Theorem~1.1]{FengLaiXiong2024} proved, in particular, that
if a positive strictly decreasing null sequence satisfies
\[
 \limsup_{n\to\infty}\frac{a_{n+1}}{a_n}<1,
\]
then it has a bi-Lipschitz image in every measurable set of positive
Lebesgue measure. This applies to every geometric sequence. A
bi-Lipschitz map need not be affine, so their result does not imply
measure universality in the sense used here. 

There are also results in which Hausdorff dimension replaces Lebesgue
measure. Keleti~\cite{Keleti2008} constructed a compact set of Hausdorff
dimension one that avoids similar copies of any prescribed countable
collection of three-point sets. Cruz, Lai and
Pramanik~\cite{CruzLaiPramanik2023} constructed Borel sets of Hausdorff
dimension one and Lebesgue measure zero avoiding affine copies of
$\{0\}\cup\{a_n:n\ge1\}$ under suitable decay assumptions. Their
result includes geometric sequences. These dimension statements do
not give a positive-measure avoiding set.

Gallagher, Lai and Weber~\cite{GallagherLaiWeber2023} studied the
corresponding question for dense $G_\delta$ sets. They proved that
Cantor sets are not topologically universal, whereas every countable
set is topologically universal. Jung and Lai~\cite{JungLai2024}
characterized topological universality by strong measure zero.
For the measure problem on Cantor sets,
Kolountzakis~\cite{Kolountzakis2023} constructed full-measure sets
avoiding certain thin Cantor sets. A different question concerns
unbounded sequences and sets having large measure in every unit
interval; see Kolountzakis and Papageorgiou~\cite{KolountzakisPapageorgiou2025}.

Our proof uses the finite-tree construction of the unpublished
manuscripts~\cite{OpenAIDyadic,OpenAIGeometric}. The second manuscript
treats $q^n$ for every fixed $q\in(0,1)$. We adapt its grids and index
intervals to a sequence with variable ratios. After passing to a
subsequence, the required bounds are
\begin{equation}\label{eq:two-sided-intro}
 0<\alpha\le\frac{b_{n+1}}{b_n}\le\beta<1.
\end{equation}
The upper bound gives separation and ensures that later translations
preserve earlier grid cells outside a small exceptional set. The lower
bound controls the number of grid boundaries crossed as the dilation
varies. The proof is included below; the geometric case and the tree
construction are due to~\cite{OpenAIDyadic,OpenAIGeometric}.

\subsection{Outline of the proof}
We construct small open $1$-periodic sets meeting
$x+t\{b_n:n\ge1\}$ for every $x\in\R$ and $t\in[1,2]$.
Dilations and reflections then cover all nonzero scales.

We assign an interval of indices, called a window, to each edge of a
finite tree. A local test uses only the tables on that edge and in its
child subtree. By choosing the window lengths, we bound the number of
representative dilations exponentially in the length. The number of
children is chosen so that the failure probability decreases faster.
Using the finest grid of the whole tree would not give this estimate.

After enlarging the random set to an open set, the exceptional centers
form a small closed set. An open neighborhood of this set meets every
remaining copy, since $x+tb_n\to x$. Thus no center is left out.

\section{Subsequence selection and periodic hitting sets}

We first select a subsequence satisfying the two-sided
bounds~\eqref{eq:two-sided-intro}.

\begin{lemma}\label{lem:subsequence}
Under the assumptions of Theorem~\ref{thm:main}, there are indices
$n_1<n_2<\cdots$ and a constant $\alpha>0$ such that
$b_j=a_{n_j}$ satisfies
\[
 \alpha\le\frac{b_{j+1}}{b_j}\le\frac12\qquad(j\ge1).
\]
\end{lemma}

\begin{proof}
Choose $c\in(0,1)$ and $N_0$ such that $a_{n+1}\ge ca_n$ for
$n\ge N_0$. Choose $n_1\ge N_0$. Once $n_j$ is chosen, let
$n_{j+1}>n_j$ be the first index for which
$a_{n_{j+1}}\le a_{n_j}/2$. Such an index exists because $a_n\to0$.
By minimality, $a_{n_{j+1}-1}>a_{n_j}/2$, and hence
\[
 \frac c2 a_{n_j}<a_{n_{j+1}}\le\frac12 a_{n_j}.
\]
The conclusion follows with $\alpha=c/2$. This argument does not use
monotonicity of $(a_n)$.
\end{proof}

For the rest of the construction, fix a positive sequence $(b_n)$
satisfying~\eqref{eq:two-sided-intro}. Then $(b_n)$ is strictly decreasing
and tends to zero. For a measurable $1$-periodic set $H\subset\R$, write
\[
 \rho(H)=m(H\cap[0,1)).
\]
We also identify $H$ with its image in $\T=\R/\Z$ when convenient.

\begin{proposition}\label{prop:periodic}
For every $\tau\in(0,1)$ there is an open $1$-periodic set $H$ with
$\rho(H)\le6\tau$ such that
\begin{equation}\label{eq:periodic-target}
 \forall x\in\R\ \forall t\in[1,2]\ \exists n\ge1:
 \qquad x+tb_n\in H.
\end{equation}
\end{proposition}

Sections~\ref{sec:grids}--\ref{sec:scales} prove this proposition.
We first deduce the main theorem from it.

\begin{proof}[Proof of Theorem~\ref{thm:main} from
Proposition~\ref{prop:periodic}]
Choose $(b_n)$ by Lemma~\ref{lem:subsequence}. For each $k\in\Z$, let
\[
 \tau_k=\frac{\eta}{64}4^{-|k|},
\]
and choose $H_k$ from Proposition~\ref{prop:periodic}. Define
\begin{equation}\label{eq:global-set}
 O=\bigcup_{k\in\Z}\bigl(2^kH_k\cup(-2^kH_k)\bigr),
 \qquad E=[0,1]\setminus O.
\end{equation}
The set $O$ is open and symmetric, so $E$ is compact. For either sign,
periodicity gives
\[
 m\bigl((\pm2^kH_k)\cap[0,1]\bigr)
 \le\max\{1,2^k\}\rho(H_k).
\]
For $k\le0$, the interval $[0,1]$ consists of an integer number of
periods. For $k>0$, it is contained in a period interval of length
$2^k$. Therefore
\begin{align*}
 m(O\cap[0,1])
 &\le12\sum_{k\in\Z}\tau_k\max\{1,2^k\}\\
 &=\frac{12\eta}{64}
   \left(\sum_{j\ge1}4^{-j}+1+\sum_{k\ge1}2^{-k}\right)
 =\frac{7\eta}{16}<\eta.
\end{align*}
Thus $m(E)>1-\eta$.

Given $x\in\R$ and $s>0$, write $s=2^kt$ with $k\in\Z$ and
$t\in[1,2)$. Apply~\eqref{eq:periodic-target} for $H_k$ at the center
$2^{-k}x$. For some $n$,
\[
 2^{-k}x+tb_n\in H_k,
 \qquad x+sb_n\in2^kH_k\subset O.
\]
For $s<0$, apply the positive-scale assertion to $-x$ and $-s$, then
use $O=-O$. Since every $b_n$ is a term of the original sequence,
\eqref{eq:main-conclusion} follows. Only the set $H_k$ was rescaled;
no relation between $(b_n)$ and powers of two is needed.
\end{proof}

\section{Grids and index windows}\label{sec:grids}

\subsection{Nested grids}
Choose $n_0$ so that $2b_{n_0}\le1/4$, and let
\[
 \gamma=\frac4{1-\beta},\qquad
 N_j=2^{\lceil\log_2(\gamma/b_j)\rceil},\qquad
 \Gamma_j=N_j^{-1}\Z\qquad(j\ge n_0).
\]
The integers $N_j$ are powers of two and are nondecreasing. They satisfy
\begin{equation}\label{eq:grid-size}
 \frac{\gamma}{b_j}\le N_j<\frac{2\gamma}{b_j}.
\end{equation}
For $z\in\R$, define the periodic cell index
\[
 J_j(z)=\lfloor N_j\{z\}\rfloor,
 \qquad \{z\}=z-\lfloor z\rfloor.
\]
Cells are closed on the left and open on the right. If $j'\ge j$, then
\begin{equation}\label{eq:nested}
 J_j(z)=\left\lfloor\frac{J_{j'}(z)}{N_{j'}/N_j}\right\rfloor.
\end{equation}
Thus the finer cell index determines every coarser one.

\subsection{Windows on a finite tree}
Let $M\ge2$, $d\ge1$, $g\ge1$, and $r_0\ge1$ be integers,
to be chosen later. Consider the complete ordered $M$-ary tree of height $d$.
Leaves have height zero. If $P$ is an internal vertex, its children
are $P_1,\ldots,P_M$, and $e=(P,i)$ denotes the edge from $P$ to $P_i$.
Let $K$ be the total number of edges.

List the edges in preorder: for each $i=1,\ldots,M$, list $(P,i)$
followed by all the edges below $P_i$, in the same recursive order.
An edge and its child subtree form one consecutive block.

Assign a window $W_e=[u_e,v_e]\cap\N$ to every edge. Consecutive
windows are separated by exactly $g$ unused indices. All edges leaving
a vertex of height $h$ have window length $r_h$. Let $\sigma_h$ be
the total index span of a subtree of height $h$, including its internal
gaps. Define
\begin{align*}
 r_1&=r_0,& \sigma_1&=Mr_1+(M-1)g,\\
 r_h&=\max\{r_0,g+\sigma_{h-1}\},&
 \sigma_h&=M(r_h+g+\sigma_{h-1})+(M-1)g
 \quad(2\le h\le d).
\end{align*}
Starting at $n_0$, place the windows in the prescribed order. If $f$
immediately precedes $e$, then $u_e=v_f+g+1$. Let
\[
 \cI=\bigcup_e W_e.
\]
Each edge uses the grid with index $v_e$. If $e=(P,i)$, let $v_e^*$
be the largest endpoint of a window on $e$ or below $P_i$.

\begin{lemma}\label{lem:span}
If $e$ leaves a vertex of height $h$, then
\begin{equation}\label{eq:span}
 v_e^*-u_e+1\le2r_h.
\end{equation}
\end{lemma}

\begin{proof}
For $h=1$, the span is $r_1$. For $h\ge2$, it is
$r_h+g+\sigma_{h-1}$, which is at most $2r_h$ by the choice of $r_h$.
\end{proof}

The bound~\eqref{eq:span} controls the finest grid in each local
test. The gaps ensure that translations from later windows preserve
earlier cell indices at the centers defined below.

\subsection{Stable centers}
A center $x\in\R$ is called stable if, for every noninitial window
$W_e$, with $v'$ the endpoint of the preceding window,
\begin{equation}\label{eq:stability-definition}
 (x,x+2b_{u_e}]\cap\Gamma_{v'}=\varnothing.
\end{equation}
Let $G$ denote the set of stable centers.

\begin{lemma}\label{lem:stable}
The set $G$ is measurable and $1$-periodic, and
\begin{equation}\label{eq:bad-density}
 \rho(\R\setminus G)\le4K\gamma\beta^{g+1}.
\end{equation}
For $x\in G$, $n\in W_e$, and $t\in[1,2]$, every edge $f$ preceding
$e$ satisfies
\begin{equation}\label{eq:earlier-indices}
 J_{v_f}(x+tb_n)=J_{v_f}(x).
\end{equation}
\end{lemma}

\begin{proof}
Failure of~\eqref{eq:stability-definition} occurs on
\[
 \bigcup_{z\in\Gamma_{v'}}[z-2b_{u_e},z).
\]
Its density is at most
\[
 2N_{v'}b_{u_e}
 \le4\gamma\frac{b_{u_e}}{b_{v'}}
 \le4\gamma\beta^{g+1},
\]
because $u_e-v'=g+1$. A union bound over the noninitial windows gives
\eqref{eq:bad-density}.

For a stable center, $0<tb_n\le2b_{u_e}$, so the translation crosses
no boundary of $\Gamma_{v'}$. The cell index is unchanged, including
when $x$ itself is a boundary: the cell convention assigns it to the
cell on its right. Equation~\eqref{eq:nested} gives
\eqref{eq:earlier-indices}. The first window has no preceding edges.
\end{proof}

\begin{lemma}\label{lem:separation}
For every edge $e$, center $x\in\R$, and $t\in[1,2]$, the points
\[
 \{x\}\cup\{x+tb_n:n\in W_e\}
\]
have pairwise distinct cell indices under $J_{v_e}$ and under every
finer grid $J_j$, $j\ge v_e$.
\end{lemma}

\begin{proof}
All these points lie in an interval of length at most
$2b_{n_0}\le1/4$. Their circular distances therefore equal their real
distances. The distance from $x$ to a translated point is at least
$b_{v_e}$. If $n<n'$ belong to $W_e$, then
\[
 t(b_n-b_{n'})\ge b_n-b_{n+1}
 \ge(1-\beta)b_n\ge(1-\beta)b_{v_e}.
\]
By~\eqref{eq:grid-size}, the cell width is at most
\[
 \frac1{N_{v_e}}\le\frac{b_{v_e}}{\gamma}
 =\frac{(1-\beta)b_{v_e}}4.
\]
Two points with the same cell index have circular distance strictly
less than this width. Their indices are therefore distinct, and
\eqref{eq:nested} preserves distinctness at finer resolutions.
\end{proof}

\section{Random paths and independent tests}\label{sec:probability}

Fix $\tau\in(0,1)$. The tree and the windows remain deterministic.
For each edge $e=(P,i)$ with $i<M$, choose a table
\[
 S_e:\{0,\ldots,N_{v_e}-1\}\longrightarrow\{0,1\}
\]
whose entries are independent Bernoulli-$1/2$ random variables.
Child $M$ is the default child and has no selector table. For each leaf $L$, let $v(L)$ be the
endpoint of its incoming window, and choose a table
\[
 T_L:\{0,\ldots,N_{v(L)}-1\}\longrightarrow\{0,1\}
\]
of independent Bernoulli-$\tau$ random variables. All table entries
are mutually independent. This gives a finite probability space with probability
$\PP$ and expectation $\EE$.

A point $z\in\R$ follows one path from the root to a leaf. At an
internal vertex $P$, it chooses the first child $P_i$, $i<M$, such that
$S_{(P,i)}(J_{v_{(P,i)}}(z))=1$. If none exists, it chooses $P_M$.
Write $L(z)$ for the leaf reached, and define
\begin{equation}\label{eq:random-set}
 B=\{z\in\R:T_{L(z)}(J_{v(L(z))}(z))=1\}.
\end{equation}
Every realization of $B$ is periodic and is a union of finitely many
half-open cells per period.

Let $\cA$ be the sigma-field generated by all selector entries.
For fixed $z$, conditioning on $\cA$ determines its leaf and terminal
address. The entry at that address still has probability $\tau$ of
being one. Consequently,
\begin{equation}\label{eq:mean-density}
 \PP(z\in B\mid\cA)=\tau,
 \qquad \EE\rho(B)=\int_0^1\PP(z\in B)\,dz=\tau.
\end{equation}
Thus increasing the number of children does not increase the expected
density of $B$.

\subsection{Conditioning at a center}
Fix $x\in\R$. In every selector table, reveal the entry whose cell
contains $x$, including tables outside the path of $x$. Let
\begin{equation}\label{eq:center-field}
 \cF_x=\sigma\bigl(S_e(J_{v_e}(x)):e=(P,i),\ i<M\bigr)
 \subset\cA.
\end{equation}
An atom $\xi$ of $\cF_x$ specifies all these values. The cells are
fixed once $x$ is fixed. Conditional on $\xi$, all other selector
entries are still independent Bernoulli-$1/2$ variables. The terminal entries are
still independent Bernoulli-$\tau$ variables.

The revealed values determine the path of $x$. To compute the law of
this path, first condition only on the choices at the strict ancestors
of a visited vertex. The $M-1$ selector entries at that vertex are
still independent Bernoulli-$1/2$ variables. The conditional probability of choosing
the default child is therefore $2^{1-M}$. Every path has $d$ steps,
so, before conditioning on an atom of $\cF_x$, we have
\begin{equation}\label{eq:no-default}
 \PP(\text{the path of }x\text{ has no default choice})
 =(1-2^{1-M})^d.
\end{equation}

Suppose that $x\in G$ and that $\xi$ determines a first default vertex
$U$ of height $h$. Write
\[
 r=r_h,\qquad e_i=(U,i)\ (1\le i<M),\qquad y_n(t)=x+tb_n.
\]
All selectors at $x$ on the edges $e_i$ are zero. For each $i<M$,
let $L_i(z)$ be the leaf reached by starting the same path rule at
$U_i$. If $U_i$ is a leaf, let $L_i(z)=U_i$. Define the local test
\begin{equation}\label{eq:local-test}
 Q_i(z)=S_{e_i}(J_{v_{e_i}}(z))\,
 T_{L_i(z)}(J_{v(L_i(z))}(z)).
\end{equation}
It uses only the edge $e_i$ and its child subtree.

\begin{lemma}\label{lem:path}
For every outcome in $\xi$, every $i<M$, $n\in W_{e_i}$, and
$t\in[1,2]$, the actual path of $y_n(t)$ reaches $U$ and rejects
children $1,\ldots,i-1$ there. In particular,
\begin{equation}\label{eq:test-hit}
 Q_i(y_n(t))=1\quad\Longrightarrow\quad y_n(t)\in B.
\end{equation}
\end{lemma}

\begin{proof}
At a strict ancestor $P$ of $U$, suppose that $x$ chooses child $P_j$.
Since $U$ is the first default, $j<M$. The decision reads zeros in the
selectors with indices less than $j$ and one in selector $j$. All
these edges precede $e_i$ in preorder. Their values at $y_n(t)$ equal those at $x$
by~\eqref{eq:earlier-indices}. Applying this at each ancestor shows
that $y_n(t)$ reaches $U$.

At $U$, the edges $e_1,\ldots,e_{i-1}$ also precede $e_i$, so their
values at $y_n(t)$ equal the exposed zeros at $x$. If the local test
is one, the actual path therefore chooses $U_i$ and then reaches
$L_i(y_n(t))$, where the terminal entry is one. This is exactly the
membership rule in~\eqref{eq:random-set}.
\end{proof}

\begin{lemma}\label{lem:independence}
Under the same conditioning, for every fixed $t\in[1,2]$,
\begin{equation}\label{eq:failure}
 \PP\bigl(Q_i(y_n(t))=0\text{ for all }i<M,\ n\in W_{e_i}
       \mid\xi\bigr)
 =(1-\tau/2)^{(M-1)r}.
\end{equation}
\end{lemma}

\begin{proof}
The index set
\[
 I=\{(i,n):1\le i<M,\ n\in W_{e_i}\}
\]
is fixed on $\xi$ and has size $(M-1)r$. For $(i,n)\in I$, let
\[
 A_{i,n}=S_{e_i}(J_{v_{e_i}}(y_n(t))).
\]
For fixed $i$, Lemma~\ref{lem:separation} shows that these selector
addresses are distinct and avoid the entry exposed at $x$. For
different $i$, the tables differ. Thus the variables $A_{i,n}$ are
independent Bernoulli-$1/2$ variables conditional on $\xi$.

Now condition on $\cA$, not just $\cF_x$. The local leaves $L_{i,n}$
and terminal addresses
\[
 L_{i,n}=L_i(y_n(t)),\qquad
 \kappa_{i,n}=\bigl(L_{i,n},J_{v(L_{i,n})}(y_n(t))\bigr)
\]
are then fixed. These addresses are pairwise distinct. Different
children of $U$ have disjoint leaf sets. Within one child subtree,
different leaves use different tables. If two tested points reach
the same leaf $L$, then $v(L)\ge v_{e_i}$, and
Lemma~\ref{lem:separation} gives different indices in its terminal
table.

The addressed terminal entries are therefore independent
Bernoulli-$\tau$ variables conditional on $\cA$. Only pairs for
which $A_{i,n}=1$ need a zero terminal entry in order to fail. Hence
\[
 \PP\bigl(Q_i(y_n(t))=0\text{ for all }(i,n)\in I\mid\cA\bigr)
 =(1-\tau)^{\sum_{(i,n)\in I}A_{i,n}}.
\]
Since $\cF_x\subset\cA$, taking conditional expectation on $\xi$
gives
\[
 \EE\left[(1-\tau)^{\sum A_{i,n}}\mid\xi\right]
 =\prod_{(i,n)\in I}\left(\frac12+\frac{1-\tau}{2}\right)
 =(1-\tau/2)^{(M-1)r}.
\]
This proves~\eqref{eq:failure}. Independence of the paths is not
needed: conditional on the selectors, the terminal addresses are
distinct and the corresponding entries are independent.
\end{proof}

\section{All scales and the exceptional centers}\label{sec:scales}

We next reduce the interval of dilations to finitely many values.
The lower ratio bound gives, for $j\ge n$,
\begin{equation}\label{eq:ratio-iterate}
 \frac{b_n}{b_j}\le\alpha^{-(j-n)}.
\end{equation}

\begin{lemma}\label{lem:all-scales}
Fix a stable center $x$ and an atom $\xi$ on which the first default
vertex $U$ has height $h$. Let $r=r_h$ and
\[
 D(r)=4+2(M-1)r\bigl(1+2\gamma\alpha^{-2r}\bigr).
\]
Then
\begin{equation}\label{eq:all-scale-bound}
 \PP\bigl(\exists t\in[1,2]\ \forall n\in\cI:
       x+tb_n\notin B\mid\xi\bigr)
 \le D(r)\exp\bigl(-\tau(M-1)r/2\bigr).
\end{equation}
\end{lemma}

\begin{proof}
The local test $Q_i(z)$ is determined by $J_{v_{e_i}^*}(z)$ for every
realization of the tables: all of its selectors and terminal tables
have resolutions at most $v_{e_i}^*$, and the grids are nested.
For $n\in W_{e_i}$, the segment traversed by $y_n(t)$ as
$t\in[1,2]$ has length $b_n$. It meets at most
\begin{equation}\label{eq:crossings}
 1+b_nN_{v_{e_i}^*}
 \le1+2\gamma\frac{b_n}{b_{v_{e_i}^*}}
 \le1+2\gamma\alpha^{-2r}
\end{equation}
boundaries of the lattice $\Gamma_{v_{e_i}^*}$. The last inequality
uses~\eqref{eq:ratio-iterate} and~\eqref{eq:span}.

Collect all these boundary scales for $i<M$ and $n\in W_{e_i}$,
together with $1$ and $2$. List their distinct values in increasing
order and add one midpoint between each consecutive pair. The
resulting finite set $\cT_{x,U}$ satisfies
\[
 \#\cT_{x,U}\le D(r).
\]
The set depends only on $x$, $U$, and the chosen grids and windows.
It is fixed on $\xi$, independently of the unrevealed entries.
Between consecutive boundary scales, all relevant cell indices are
constant. Boundary scales are included separately, so no endpoint
is lost under the half-open cell convention. For every $t\in[1,2]$
there is $t'\in\cT_{x,U}$ with
\[
 Q_i(y_n(t))=Q_i(y_n(t'))\qquad(i<M,\ n\in W_{e_i}).
\]

If $x+tb_n\notin B$ for all $n\in\cI$, then
Lemma~\ref{lem:path} forces all the local tests to fail. They also
fail at a representative $t'$. Apply Lemma~\ref{lem:independence}
to each representative and take a union bound. Since
$1-\tau/2\le e^{-\tau/2}$, this gives~\eqref{eq:all-scale-bound}.
This union bound does not require independence between different
representative scales.
\end{proof}

\begin{proof}[Proof of Proposition~\ref{prop:periodic}]
Choose the parameters in the following order. First take $M$ so that
\begin{equation}\label{eq:branch-choice}
 \lambda:=\frac{\tau(M-1)}2>2\log(1/\alpha).
\end{equation}
Choose $d$ so that $(1-2^{1-M})^d<\tau$. This fixes the tree and
its number $K$ of edges. Next choose $g$ so that
\[
 4K\gamma\beta^{g+1}<\tau.
\]
Finally,
\[
 D(r)e^{-\lambda r}
 =O\left(r e^{-(\lambda-2\log(1/\alpha))r}\right)
 \longrightarrow0.
\]
Choose $r_0$ large enough that $D(r)e^{-\lambda r}<\tau$ for every
integer $r\ge r_0$, and construct the windows of
Section~\ref{sec:grids}. The tree and $g$ remain fixed as the
window lengths increase. Every window length is at least $r_0$, so
the last inequality applies to each window.

For a fixed stable center $x$, equation~\eqref{eq:no-default} bounds
the probability of no default by $\tau$. On every atom of $\cF_x$ where
a first default exists, Lemma~\ref{lem:all-scales} bounds the
conditional probability of a missed scale by $\tau$. Thus
\begin{equation}\label{eq:fixed-center}
 \PP\bigl(\exists t\in[1,2]\ \forall n\in\cI:
       x+tb_n\notin B\bigr)\le2\tau
 \qquad(x\in G).
\end{equation}
The probability space is finite. Hence the event at each fixed center
is measurable, although it involves every $t\in[1,2]$.

For each outcome $\omega$, enlarge $B_\omega$ to an open periodic set
$B_\omega^+$ satisfying
\[
 B_\omega\subset B_\omega^+,
 \qquad \rho(B_\omega^+)\le\rho(B_\omega)+\tau.
\]
There are only finitely many cells in one period. We may add their
endpoints and then take sufficiently small open neighborhoods in
$\T$. Define
\begin{equation}\label{eq:exceptional-set}
 R_\omega=\{x\in\R:\exists t\in[1,2]\ \forall n\in\cI,
                     \ x+tb_n\notin B_\omega^+\}.
\end{equation}
This set is periodic and closed. To see closedness, work on $\T$.
For each $n\in\cI$, the map $(x,t)\mapsto x+tb_n\pmod1$ is
continuous, and $\T\setminus B_\omega^+$ is closed. The pairs
$(x,t)$ in~\eqref{eq:exceptional-set} therefore form a closed subset
of the compact space $\T\times[1,2]$. Its projection onto $\T$ is
compact, and hence closed.

For $x\in G$, inclusion $B_\omega\subset B_\omega^+$ and
\eqref{eq:fixed-center} give $\PP(x\in R_\omega)\le2\tau$.
At other centers this probability is at most one. The finite
probability space and closedness of $R_\omega$ justify integration,
and Lemma~\ref{lem:stable} gives
\[
 \EE\rho(R_\omega)
 \le2\tau\rho(G)+\rho(\R\setminus G)\le3\tau.
\]
Also, by~\eqref{eq:mean-density}, $\EE\rho(B_\omega^+)\le2\tau$.
Choose an outcome for which
\[
 \rho(B^+)+\rho(R)\le5\tau.
\]
Outer regularity on $\T$ gives an open periodic set $V\supset R$
with $\rho(V)\le\rho(R)+\tau$. Then
\[
 H=B^+\cup V
 \qquad\text{satisfies}\qquad \rho(H)\le6\tau.
\]

Fix $x\in\R$ and $t\in[1,2]$. If $x\notin R$, the definition
of $R$ gives some $n\in\cI$ with $x+tb_n\in B^+$. If $x\in R$,
then $x\in V$ and $V$ is open, so $x+tb_n\in V$ for all sufficiently
large $n$, since $b_n\to0$. The latter indices need not belong to
$\cI$. Both cases prove~\eqref{eq:periodic-target}.
\end{proof}

\section{Consequences and the role of the hypothesis}

For a countable family, apply the main theorem to each tail with
summable bounds on the measure removed.

\begin{corollary}\label{cor:countable}
Let $(a_n^{(j)})_{n\ge1}$, $j\ge1$, be a countable family of positive
null sequences, each satisfying~\eqref{eq:main-hypothesis}. For every
$\eta\in(0,1)$ there is a compact $E\subset[0,1]$ with $m(E)>1-\eta$
such that, for every $j$, $x\in\R$, and $s\ne0$, infinitely many
distinct points of $\{x+sa_n^{(j)}:n\ge1\}$ lie outside $E$.
\end{corollary}

\begin{proof}
For each $j,N\ge1$, apply Theorem~\ref{thm:main} to the tail
$(a_n^{(j)})_{n\ge N}$ with $\eta$ replaced by $\eta2^{-j-N-1}$. Let $E_{j,N}$
be the resulting compact set and define $E=\bigcap_{j,N\ge1}E_{j,N}$.
Then $E$ is compact and
\[
 m([0,1]\setminus E)
 \le\sum_{j,N\ge1}\eta2^{-j-N-1}=\frac\eta2<\eta.
\]
For fixed $j,x,s$ and every $N$, some $n\ge N$ satisfies
$x+sa_n^{(j)}\notin E_{j,N}$ and hence $x+sa_n^{(j)}\notin E$.
There are infinitely many such indices. They give infinitely many
distinct points, because a positive null sequence takes each fixed
positive value only finitely often.
\end{proof}

For example, Theorem~\ref{thm:main} applies when
\[
 a_n=a_1\prod_{k=1}^{n-1}r_k,
 \qquad r_k\in\{1/3,1/2\},
\]
with an arbitrary choice at each stage. It also applies to
$a_n=e^{-n-\sqrt n}$, whose adjacent ratios tend to $e^{-1}$.
The case $a_n=e^{-n^\theta}$ with $0<\theta<1$ is included, although
that case already follows from the results for ratios tending to one.

The use of the lower ratio bound in~\eqref{eq:crossings} can also
be seen by writing $u_n=-\log b_n$. The two-sided bounds give
\[
 \log(1/\beta)\le u_{n+1}-u_n\le\log(1/\alpha).
\]
Across at most $2r$ consecutive steps, $u_n$ increases by at most
$2r\log(1/\alpha)$. This bounds the number of representative dilations
by an exponential in $r$. The failure probability decreases faster
by our choice of $M$.

The proof does not give the same conclusion for every sequence with
$\liminf a_{n+1}/a_n=0$. For instance, if $b_n=2^{-2^n}$, then
\[
 \frac{b_n}{b_m}=2^{2^m-2^n}\qquad(m>n),
\]
so~\eqref{eq:ratio-iterate} fails for every fixed $\alpha>0$.
This is a limitation of the estimate, not a claim of measure universality.

To settle the full conjecture, it would suffice to treat positive null
sequences with $a_{n+1}/a_n\to0$. Indeed, every positive null sequence
has a subsequence $(c_j)$ with $c_{j+1}\le2^{-j}c_j$, and avoiding a
subset suffices to avoid the whole set. After translation and possibly
reflection, every bounded infinite set contains a positive null
sequence. Unbounded sets are avoided by any bounded set of positive
measure.

\end{document}